\documentclass[10pt,a4paper,reqno]{amsart}
 \theoremstyle{plain}
 \newtheorem{tm}{Theorem}
 \newtheorem{lm}[tm]{Lemma}
\newtheorem{cor}[tm]{Corollary}
 \newtheorem{prop}[tm]{Proposition}

 \theoremstyle{definition}

  \newtheorem{ex}[tm]{Example}

 \newcommand{\bN}{\ensuremath{\mathbb N}}
 \newcommand{\bZ}{\ensuremath{\mathbb Z}}
  
  \newcommand{\pf}{\mathop{\rm pf}\nolimits}

 \usepackage{graphics}
 \usepackage[english]{babel}

\usepackage[latin1]{inputenc}

\usepackage{times}
\usepackage[T1]{fontenc}
 \usepackage{graphics}
\usepackage{pst-all}

\newcommand{\dis}{\displaystyle}
\def\a{{\alpha}}

 \author{Valentina Barucci, Ralf Fr\"oberg, Mesut \c{S}ah\.{i}n}

 \thanks{The third author is supported by T\"{U}B\.{I}TAK-2219}

\begin{document}

\title{On free resolutions of some semigroup rings}
\begin{abstract}
For some numerical semigroup rings of small embedding dimension, namely those of embedding dimension 3, and symmetric or pseudosymmetric of embedding dimension 4, presentations has been determined in the literature. We extend these results  by giving the whole graded minimal free resolutions explicitly. Then we use these resolutions to determine some invariants of the
semigroups and  certain interesting relations among them. Finally, we determine semigroups of small embedding dimensions which have  strongly indispensable resolutions.
\end{abstract}

 \maketitle
 
\section{Introduction}
Let $S=\langle n_1,\ldots,n_k\rangle$ be a numerical semigroup, i.e., $n_i$ are positive integers with greatest common divisor $1$, and $S=\{\sum_{i=1}^ku_in_i : u_i \; \mbox{are nonnegative integers}\}$. Let $PF(S)=\{ n\in\bZ \setminus S:n+s\in S \; \mbox{for all} \; s\in S\setminus \{0\}\}$. The elements in $PF(S)$ are called the pseudofrobenius numbers
of $S$. Since $S$ is a numerical semigroup, $\bN\setminus S$ is finite. The largest integer $g(S)\notin S$ belongs to $PF(S)$ and is called
the Frobenius number of $S$. If $PF(S)=\{ g(S)\}$, $S$ is called symmetric, since then, for each $n\in\bZ$, exactly one of $n$ and $g(S)-n$ lies in $S$.
If $PF(S)=\{ g(S)/2,g(S)\}$, $S$ is called pseudosymmetric. Let $K$ be a field and $K[S]=K[t^{n_1},\ldots,t^{n_k}]$   be the semigroup ring of $S$, then $K[S]\simeq A/I_S$ where, $A=K[x_1,\ldots,x_k]$ and $I_S$ is the kernel of the surjection
$A \stackrel{\phi_0}{\longrightarrow} K[S]$, where $x_i\mapsto t^{n_i}$. If $\deg_S(x_i)=n_i$, this map is homogeneous of degree $0$. Throughout the paper, we drop $S$ in the notation and simply use $\deg(F)$ for a polynomial $F\in A$, except in the proof of Theorem \ref{thm4ci} where there are two semigroups involved.
$S$ is symmetric if and only if $k[S]$ is a Gorenstein ring \cite{ku}. If the embedding dimension $k$ is small, then $I_S$
has been determined in some cases, e.g. if $k=3$  by Herzog \cite{he}, if  $k=4$ and $S$ symmetric by Bresinsky \cite{br}, and if $k=4$ and $S$ pseudosymmetric by Komeda \cite{ko}. We will
determine a minimal graded $A$-resolution in these cases. In the first case the resolution was given by Denham \cite{den}. There is a concept, strong indispensability, which give a kind of uniqueness of the
minimal graded resolution. In the last section we classify semigroup rings of small embedding dimension which have strongly indispensable resolutions. The original motivation for strong indispensability comes from its applications in Algebraic Statistics, see e.g. \cite{tak}.

\section{Resolutions}
For completeness we start with $3$-generated symmetric semigroups. If $S$ is symmetric, then $K[S]$ is a complete intersection (\cite[Theorem 3.10]{he}), so the resolution is given by the Koszul complex. If $S$ is not symmetric, then we use Herzog's result.

\begin{tm}\cite[Proposition 3.2]{he}
Let $\alpha_i$, $1\le i\le 3$ be the smallest positive integer such that $\alpha_in_i\in\langle n_k,n_l\rangle$, $\{ i,k,l\}=\{ 1,2,3\}$,
and let $\alpha_in_i=\alpha_{ik}n_k+\alpha_{il}n_l$. Then $S=\langle n_1,n_2,n_3\rangle$ is $3$-generated not symmetric if and only if $\alpha_{ik}>0$ for all $i,k$,
$\alpha_{21}+\alpha_{31}=\alpha_1,\alpha_{12}+\alpha_{32}=\alpha_2,\alpha_{13}+\alpha_{23}=\alpha_3$. Then
$$K[S]=A/(f_1,f_2,f_3)$$

where $f_1=x_1^{\alpha_1}-x_2^{\alpha_{12}}x_3^{\alpha_{13}}$, $f_2=x_2^{\alpha_2}-x_1^{\alpha_{21}}x_3^{\alpha_{23}}$, $f_3=x_3^{\alpha_3}-x_1^{\alpha_{31}}x_2^{\alpha_{32}}.$
\end{tm}
Denham gave a minimal graded $A$-resolution of $K[S]$.

\begin{tm}\label{3generated}\cite[Lemma 2.5]{den}
If $S$ is a $3$-generated semigroup which is not symmetric then
$K[S]$ has a minimal graded free $A$-resolution
$$0\longrightarrow A^2\stackrel{\phi_2}{\longrightarrow} A^3\stackrel{\phi_1}{\longrightarrow}A\longrightarrow 0,$$
where $\phi_1=(x_1^{\alpha_1}-x_2^{\alpha_{12}}x_3^{\alpha_{13}},x_2^{\alpha_2}-x_1^{\alpha_{21}}x_3^{\alpha_{23}},x_3^{\alpha_3}-x_1^{\alpha_{31}}x_2^{\alpha_{32}})$,
and $\phi_2=\left( \begin{array}{cc}x_3^{\alpha_{23}}&x_2^{\alpha_{32}}\\x_1^{\alpha_{31}}&x_3^{\alpha_{13}}\\x_2^{\alpha_{12}}&x_1^{\alpha_{21}}\end{array}\right).$
\end{tm}

\bigskip
Next we look at $4$-generated symmetric but not complete intersection semigroups.
We use Bresinsky's theorems.
\begin{tm}\cite[Theorem 5, Theorem 3] {br}
The semigroup $S$ is $4$-generated symmetric, not complete intersection, if and only if there are integers $\alpha_i$,  $1\le i\le4$,
$\alpha_{ij},ij\in\{21,31,32,42,13,43,14,24\}$, such that
$0<\alpha_{ij}<\alpha_i$, for all $i,j$, 
\begin{eqnarray*} \alpha_1=\alpha_{21}+\alpha_{31}, \alpha_2=\alpha_{32}+\alpha_{42},
\alpha_3=\alpha_{13}+\alpha_{43}, \alpha_4=\alpha_{14}+\alpha_{24}\;\mbox{ and}
\end{eqnarray*}
\begin{eqnarray*} n_1=\alpha_2\alpha_3\alpha_{14}+\alpha_{32}\alpha_{13}\alpha_{24},&\quad
n_2=\alpha_3\alpha_4\alpha_{21}+\alpha_{31}\alpha_{43}\alpha_{24}, \\
n_3=\alpha_1\alpha_4\alpha_{32}+\alpha_{14}\alpha_{42}\alpha_{31},&\quad
n_4=\alpha_1\alpha_2\alpha_{43}+\alpha_{42}\alpha_{21}\alpha_{13}. 
\end{eqnarray*}
Then, $K[S]=A/(f_1,f_2,f_3,f_4,f_5)$, where
\begin{eqnarray*}f_1=x_1^{\alpha_1}-x_3^{\alpha_{13}}x_4^{\alpha_{14}},& \quad 
f_2=x_2^{\alpha_2}-x_1^{\alpha_{21}} x_4^{\alpha_{24}}, & \quad
f_3=x_3^{\alpha_3}-x_1^{\alpha_{31}}x_2^{\alpha_{32}},\\
f_4=x_4^{\alpha_4}-x_2^{\alpha_{42}}x_3^{\alpha_{43}},& \quad\quad \quad
f_5=x_3^{\alpha_{43}}x_1^{\alpha_{21}}-x_2^{\alpha_{32}}x_4^{\alpha_{14}}.
\end{eqnarray*}
\end{tm}

We now give the whole minimal graded free $A$-resolution of $K[S]$ such that the matrix representation of $\phi_2$ with respect to a suitable basis of $A^5$ is an alternate matrix whose pfaffians give $\phi_1$ and $\phi_3$. The structure of the resolution is known by \cite{be} and our main contribution is to give the matrix $\phi_2$ explicitly. The proof will follow in the next section.

\begin{tm} \label{4generated} If $S$ is a $4$-generated symmetric, not a complete intersection, semigroup, then the following is a minimal graded free $A$-resolution of $K[S]$:

$$0\longrightarrow A\stackrel{\phi_3}{\longrightarrow}A^5\stackrel{\phi_2}{\longrightarrow}A^5\stackrel{\phi_1}{\longrightarrow}A\longrightarrow 0$$
where $\phi_1=(f_1,f_2,f_3,f_4,f_5)$

$$\phi_2=\left( \begin{array}{ccccc} 
0&-x_3^{\alpha_{43}}&0&-x_2^{\alpha_{32}}&-x_4^{\alpha_{24}}\\
&\\x_3^{\alpha_{43}}&0&x_4^{\alpha_{14}}&0&-x_1^{\alpha_{31}}\\
&\\0&-x_4^{\alpha_{14}}&0&-x_1^{\alpha_{21}}&-x_2^{\alpha_{42}}\\
&\\x_2^{\alpha_{32}}&0&x_1^{\alpha_{21}}&0&-x_3^{\alpha_{13}}\\
&\\x_4^{\alpha_{24}}&x_1^{\alpha_{31}}&x_2^{\alpha_{42}}&x_3^{\alpha_{13}}&0
\end{array}\right)$$

\bigskip

\noindent and $\phi_3=\phi_1^t$.
\end{tm}

Next we look at $4$-generated pseudosymmetric semigroups.
We use Komeda's theorems.
\begin{tm}\cite[Theorem 6.5, Theorem 6.4]{ko} \label{kom} The semigroup 
$S$ is $4$-generated pseudosymmetric if and only if there are positive integers $\alpha_i$,
$1\le i\le4$, and $\alpha_{21}$, with $\alpha_{21}<\alpha_1$,
such that $n_1=\alpha_2\alpha_3(\alpha_4-1)+1$,
$n_2=\alpha_{21}\alpha_3\alpha_4+(\alpha_1-\alpha_{21}-1)(\alpha_3-1)+\alpha_3$,
$n_3=\alpha_1\alpha_4+(\alpha_1-\alpha_{21}-1)(\alpha_2-1)(\alpha_4-1)-\alpha_4+1$,
$n_4=\alpha_1\alpha_2(\alpha_3-1)+\alpha_{21}(\alpha_2-1)+\alpha_2$. Then,
$K[S]=A/(f_1,f_2,f_3,f_4,f_5)$, where
\begin{eqnarray*} f_1&=&x_1^{\alpha_1}-x_3x_4^{\alpha_4-1},  \quad \quad
f_2=x_2^{\alpha_2}-x_1^{\alpha_{21}}x_4, \quad
f_3=x_3^{\alpha_3}-x_1^{\alpha_1-\alpha_{21}-1}x_2,\\
f_4&=&x_4^{\alpha_4}-x_1x_2^{\alpha_2-1}x_3^{\alpha_3-1}, \quad 
f_5=x_3^{\alpha_3-1}x_1^{\alpha_{21}+1}-x_2x_4^{\alpha_4-1}.
\end{eqnarray*}
\end{tm}

Note that Komeda calls these semigroups almost symmetric.

\medskip
We now give the whole minimal $A$-resolution of $K[S]$. The proof will follow in the next section.

\begin{tm} \label{4generatedpseudo}If $S$ is a $4$-generated pseudosymmetric semigroup, then the following is a minimal graded free $A$-resolution of $K[S]$:
$$0\longrightarrow A^2\stackrel{\phi_3}{\longrightarrow}A^6\stackrel{\phi_2}{\longrightarrow}A^5\stackrel{\phi_1}{\longrightarrow}A\longrightarrow 0$$
where $\phi_1=(f_1,f_2,f_3,f_4,f_5),$

$$\phi_2=\left( \begin{array}{cccccc}x_2&0&x_3^{\alpha_3-1}&0&x_4&0\\
0&f_3&0&x_1x_3^{\alpha_3-1}&x_1^{\alpha_1-\alpha_{21}}&x_4^{\alpha_4-1}\\
x_1^{\alpha_{21}+1}&-f_2&x_4^{\alpha_4-1}&0&x_1x_2^{\alpha_2-1}&0\\
0&0&0&x_2&x_3&x_1^{\alpha_{21}}\\
-x_3&0&-x_1^{\alpha_1-\alpha_{21}-1}&x_4&0&x_2^{\alpha_2-1}\\
\end{array}\right),
$$

\bigskip

\noindent and $\phi_3=\left( \begin{array}{cccccc}x_4&-x_1&0&x_3&-x_2&0\\
-x_2^{\alpha_2-1}x_3^{\alpha_3-1}&x_4^{\alpha_4-1}&f_2&-x_1^{\alpha_1-1}&x_1^{\alpha_{21}}x_3^{\alpha_3-1}&f_3
\end{array}\right)^t.$
\end{tm}

\section{Proofs}
In all proofs we use the following theorem by Buchsbaum-Eisenbud, see also \cite{ei}.
\begin{tm}\cite[Corollary 2]{bu-ei}
Let
$$0\longrightarrow F_n\stackrel{\phi_n}{\longrightarrow}F_{n-1}\stackrel{\phi_{n-1}}{\longrightarrow}\cdots\stackrel{\phi_2}{\longrightarrow}F_1\stackrel{\phi_1}{\longrightarrow}F_0$$
be a complex of free modules over a Noetherian ring $A$. Let {\rm rank}$({\phi_i})$ be the size of the largest nonzero minor in the matrix describing $\phi_i$, and let $I(\phi_i)$ be the
ideal generated by the minors of maximal rank. Then the complex is exact if and only if for all $1\le i\le n$

(a) $\mbox {\rm rank}(\phi_{i+1})+\mbox {\rm rank}(\phi_i)=\mbox {\rm rank}(F_i)$ and

(b) $I(\phi_i)$ contains an $A$-sequence of length $i$.
\end{tm}

In all theorems it is an easy, but sometimes tedious, task to check that we have complexes. We consider this done.

\bigskip\noindent
{\bf Proof of Theorem~\ref{3generated}.} For completeness we give the proof also in this case.
We have to show
that rank$(\phi_1)=1$ and rank$(\phi_2)=2$.
Furthermore that $I(\phi_i)$ contains a regular sequence of length $i$ for $i=1,2$. Since $I(\phi_1)=I(\phi_2)=(f_1,f_2,f_3)$, and
$K[S]$ is 1-dimensional Cohen-Macaulay, this is clear.\hfill$\Box$

\bigskip\noindent
{\bf Proof of Theorem~\ref{4generated}.}
We have to show
that the rank$(\phi_1)=$rank$(\phi_3)=1$, and that rank$(\phi_2)=4$.
Furthermore that $I(\phi_i)$ contains a regular sequence of length $i$ for all $1\leq i \leq 3$.
That rank$(\phi_1)=\mbox {rank}(\phi_3)=1$ is clear. Denote by $\pf(\Delta)$ the pfaffian of $\Delta$ and by $\Delta_{ij}$ the matrix obtained from $\Delta$ by deleting the $i$-th row and $j$-th column. Setting $\Delta=\phi_2$ for notational convenience, we observe that $\pf (\Delta_{ii})=f_i$ for $i=1,3,5$ and $\pf (\Delta_{ii})=-f_i$ for $i=2,4$. As $\det(\Delta_{ii})=[\pf(\Delta_{ii})]^2$, we get $\det(\Delta_{11})=f_1^2$ and $\det(\Delta_{22})=f_2^2$. These two determinants are relatively prime, so they constitute a regular sequence. In fact, $I(\phi_2)=(f_1,f_2,f_3,f_4,f_5)^2$.\hfill$\Box$

\bigskip\noindent
{\bf Proof of Theorem~\ref{4generatedpseudo}.}
We have to show
that rank$(\phi_1)=1$, rank$(\phi_2)=4$, and that rank$(\phi_3)=2$.
Furthermore that $I(\phi_i)$ contains a regular sequence of length $i$ for all $1\le i\le3$.
That rank$(\phi_1)=1$ is clear. Of course rank$(\phi_2)\le5$, but $\phi_2$ has a nonzerodivisor in the kernel, so by McCoy's theorem rank$(\phi_2)\le4$.
Among the 4-minors of $\phi_2$ we have $x_3f_3^3$ and $x_2f_2f_4$.
They are relatively prime, so $I(\phi_2)$ contains a regular sequence of length 2.
The following elements are 2-minors of $\phi_3$: $f_1,f_4,f_5,x_3f_2,x_3f_3$.
Since $f_1$ and $f_4$ are relatively prime, they constitute a regular sequence.
Since $k[S]$ is 1-dimensional Cohen-Macaulay, $(f_1,f_2,f_3,f_4,f_5)$ contains a regular sequence of length four,
at least one of $f_1,f_i,f_4$, $i\in\{ 2,3,5\}$ is a regular sequence.
It is clear that $x_3$ is a nonzerodivisor mod $(f_1,f_4)$. If $f_1,f_2,f_4$ is a regular sequence, then $f_1,x_3f_2,f_4$ is. If $f_1,f_3,f_4$ is a regular sequence, then $f_1,x_3f_3,f_4$ is. Otherwise  $f_1,f_2,f_5$ is a regular sequence. Thus $I(\phi_3)$ contains a regular sequence of length 3.\hfill$\Box$


\section{Applications}
We will use the following well known facts:
 If the numerical semigroup $S$ is generated by $k$ elements, and $A=K[x_1, \dots,x_k]$,  then the free minimal $A$-resolution of $K[S]$ has length codim$(K[S])=k-1$, since $K[S]$ is  a 1-dimensional Cohen-Macaulay ring:
 $${\bf F}: \ 0\longrightarrow A^{\beta_{k-1}}\stackrel{\phi_{k-1}}{\longrightarrow}A^{\beta_{k-2}}\stackrel{\phi_{k-2}}{\longrightarrow}\cdots\stackrel{\phi_2}{\longrightarrow}A^{\beta_1}\stackrel{\phi_1}{\longrightarrow}A^{\beta_0}{\longrightarrow}K[S]{\longrightarrow}0. $$

\noindent The alternating sum of the $\beta_i$'s, the Betti numbers, is zero, where $\beta_i=\dim_KH_i({\bf F}\otimes K)$ is defined by $i$-th homology of the tensored complex. On the other hand, the Betti numbers of $R:=A/I_S$ are defined by
$\beta_i=\dim_K{\rm Tor}_i^A(R,K).$ As $ R \cong K[S]$, we identify them and this gives us an alternative way
to define the Betti numbers, since also ${\rm Tor}_i^A(R,K)=H_i({\bf G}\otimes R)$, where ${\bf G}$ is a minimal $A$-resolution
of $K$ (the Koszul complex). As $R$ is Cohen-Macaulay, the highest nonzero Betti number is called the CM-type of $R$.
The ring $R$ is homogeneous if we set $\deg(x_i)=n_i$. If we concentrate ${\bf F}$ above to a certain degree $d$, we get an exact sequence of vector spaces
$$0\longrightarrow\oplus_j(A[-j]^{\beta_{k-1,j}})_d\longrightarrow\cdots\longrightarrow\oplus_j(A[-j]^{\beta_{1,j}})_d\longrightarrow A_d\longrightarrow(R)_d
\longrightarrow 0$$
so
$$0\longrightarrow\oplus_jA^{\beta_{k-1,j}}_{d-j}\longrightarrow\cdots\longrightarrow\oplus_jA^{\beta_{1,j}}_{d-j}\longrightarrow A_d\longrightarrow(R)_d
\longrightarrow 0$$
where the $\beta_{i,j}$ are the graded Betti numbers of $R=K[S]$.
The alternating sum of the dimensions of these vector spaces is 0. Multiplying each dimension with $z^d$ and summing for $d\ge0$, we get
$${\rm Hilb}_{R}(z)={\rm Hilb}_A(z)(1+\sum_{i=1}^{k-1}\sum_j(-1)^i\beta_{i,j}z^j).$$ 
Letting ${\mathcal K}_S(z)=1+\sum_{i=1}^{k-1}\sum_j(-1)^i\beta_{i,j}z^j$ and using Hilb$_{A}(z)=1/\prod_{i=1}^k(1-z^{n_i})$, we observe that
$$\frac{{\mathcal K}_S(z)}{\prod_{i=1}^k(1-z^{n_i})}={\rm Hilb}_{R}(z)={\rm Hilb}_{K[S]}(z)=\sum_{s\in S}z^s.$$

\smallskip
Recall that  the set of pseudofrobenius numbers of a numerical semigroup $S$ is $PF(S)=\{ n\in \bZ\setminus S: n+s\in S\mbox{ for all } s\in S\setminus\{0\}\}$ and its cardinality is by definition the type of the semigroup $S$. It is known that the type of $S$ coincides with the CM-type of the semigroup ring $K[S]$. We show here below how this relation is more strict.

\begin{lm}\label{soc}
Let $S=\langle n_1,\ldots,n_k\rangle$, $0\neq s\in S$, and $K[S]=K[t^{n_1},\ldots,t^{n_k}]$. Then $n\in PF(S)$ if and only if
$\overline{0}\ne\overline{t^{n+s}}\in{\rm Soc}(K[S]/(t^s))$.
\end{lm}

\begin{proof} Let $M=(t^{n_1},\ldots,t^{n_k})$. We have $n\in PF(S)$ if and only if $t^n\notin K[S]$ and $t^nM\subseteq K[S]$, so
if and only if $t^{n+s}\notin t^sK[S]$ and $t^{n+s}M\subseteq t^sK[S]$, so if and only if $ \overline{t^{n+s}}\ne\overline{0}$ and $\overline{t^{n+s}M}=\overline{0}$
in $K[S]/(t^s)$, so if and only if $\overline{0}\neq\overline{t^{n+s}}\in{\rm Soc}(K[S]/(t^s))$.
\end{proof}

\begin{prop}\label{bettisoc} Let $S=\langle n_1,\dots,n_k\rangle $ and let $\beta_{i,j}$ be the graded Betti numbers of $K[S]$. Then
 $n\in PF(S)$ if and only if $\beta_{k-1,n+N}\ne0$ (in fact $\beta_{k-1,n+N}=1$), where $N=\sum_{i=1}^kn_i$.
In particular, $S$ is symmetric if and only if $\beta_{k-1}=\beta_{k-1,g(S)+N}$.
\end{prop}

\begin{proof} Let $s=n_1$. Then the dimension of $H_{k-1}({\bf H})$, where ${\bf H}$ is a graded resolution of $K[S]/(t^{n_1})$, is the highest
nonzero Betti number $\beta_{k-1}$ of $K[S]$, and equals the dimension of ${\rm Soc}(K[S]/(t^{n_1}))$,
which exists in degrees $$n_2+\cdots+n_k+\deg{\rm Soc}(K[S]/(t^{n_1})).$$ Thus,  by Lemma \ref{soc},   $n\in PF(S)$
if and only if $\beta_{k-1,n+N}\ne0$ (in fact $\beta_{k-1,n+N}=1$, corresponding to the Frobenius number).
\end{proof}

We illustrate the proposition with an example.

\begin{ex} The semigroup $S=\langle 7,9,8,13\rangle$ is symmetric and not complete intersection by Theorem 3, thus the ring $R=K[S]$ is Gorenstein and not
a complete intersection. Set $\bar R=R/(t^7)$. The dimension of Soc$(\bar R)$ is one since $\bar R$ is also a Gorenstein ring and by Lemma 8 it is
generated by $\overline{t^{g(S)+7}}=\overline{t^{26}}$. Since ${\bf G}$ is the Koszul complex of length $k-1=3$ in the three variables $x_2,x_3,x_4$ of degrees
$n_2,n_3,n_4$, the vector space $H_3({\bf G}\otimes R)$ is nonzero only in degree $(g(S)+n_1)+(n_2+n_3+n_4)=(19+7)+(9+8+13)=56$.
\end{ex}

\begin{cor} In the notation of Theorem $1$, if $S=\langle n_1,n_2,n_3 \rangle$ is not symmetric, then $PF(S)=\{ \alpha_1n_1+\alpha_{23}n_3-N,\alpha_1n_1+\alpha_{32}n_2-N\}$, where $N=n_1+n_2+n_3$.  \end{cor}

 \begin{proof} We get $\beta_2=\beta_{2, \alpha_1n_1+\alpha_{23}n_3}+\beta_{2,\alpha_1n_1+\alpha_{32}n_2}$ by adding the degrees
 in the resolution given in Theorem \ref{3generated} and by using Proposition \ref{bettisoc}.
 \end{proof}

 \smallskip This corollary extends the result in \cite[Corollary 12]{ro-ga}, where the Frobenius number of $3$-generated semigroups is determined.

 \begin{ex}  Let $S=\langle 7,9,10\rangle$. Then, $S$ is $3$-generated not symmetric as
 $$\alpha_1=4,\alpha_{12}=2,\alpha_{13}=1,\alpha_2=3,\alpha_{21}=1,\alpha_{23}=2,\alpha_3=3,\alpha_{31}=3,\alpha_{32}=1.$$ We have, by Theorem 2, $\beta_{1,i}\ne0$ (in fact $\beta_{1,i}=1$) only if $i\in\{\alpha_1 n_1, \alpha_2n_2, \alpha_3n_3\}=\{ 28,27,30\}$, and
 $\beta_{2,i}\ne0$ (in fact $\beta_{2,i}=1$) only if $i\in\{ \alpha_1n_1+\alpha_{23}n_3,\alpha_1n_1+\alpha_{32}n_2\}=\{28+20,28+9\}=\{48,37\}$.
 Thus $PF(S)=\{48-N,37-N\}=\{22,11\}$ and we obtain the $\mathcal K$-polynomial as
 $${\mathcal K}_s=1-z^{28}-z^{27}-z^{30}+z^{48}+z^{37}$$
 so that
 $$\sum_{s\in S}z^s=\frac{{\mathcal K}_S(z)}{(1-z^7)(1-z^9)(1-z^{10})}.$$
 \end{ex}
 
  \begin{cor}\label{cor4} If $S$ is $4$-generated symmetric, not a complete intersection, we always have
 \begin{eqnarray*}a_1&=&\alpha_2n_2+\alpha_{43}n_3=\alpha_4n_4+\alpha_{32}n_2=\alpha_{21}n_1+\alpha_{43}n_3+\alpha_{24}n_4\\
 a_2&=&\alpha_1n_1+\alpha_{43}n_3=\alpha_3n_3+\alpha_{14}n_4=\alpha_{32}n_2+\alpha_{14}n_4+\alpha_{31}n_1\\
 a_3&=&\alpha_2n_2+\alpha_{14}n_4=\alpha_4n_4+\alpha_{21}n_1=\alpha_{21}n_1+\alpha_{43}n_3+\alpha_{42}n_2\\
 a_4&=&\alpha_1n_1+\alpha_{32}n_2=\alpha_3n_3+\alpha_{21}n_1=\alpha_{32}n_2+\alpha_{14}n_4+\alpha_{13}n_3\\
 a_5&=&\alpha_1n_1+\alpha_{24}n_4=\alpha_2n_2+\alpha_{31}n_1=\alpha_3n_3+\alpha_{42}n_2=\alpha_4n_4+\alpha_{13}n_3
 \end{eqnarray*}
  and
 $$a_1+\alpha_1n_1=a_2+\alpha_2n_2=a_3+\alpha_3n_3=a_4+\alpha_4n_4=a_5+\alpha_{21}n_1+\alpha_{43}n_3.$$
 \end{cor}

 \begin{proof} If we multiply $\phi_1$ with the first column of $\phi_2$ we get $f_2x_3^{\alpha_{43}}+f_4x_2^{\alpha_{32}}+f_5x_4^{\alpha_{24}}$. Since the resolution is graded, these three terms have the same degree $a_1$ which is the inner degree where we have $\beta_{2,a_1}=1$. Thus we get the equalities for $a_1$, and the equalities for $a_2,a_3,a_4$, and $a_5$ are proved similarly. The last line is the inner degree of the last module in the resolution and we get it by comparing the degrees when we multiply $\phi_2$ with $\phi_3$.
 \end{proof}

\begin{cor}
In the notation of Theorem $3$, if $S=\langle n_1,\dots,n_4\rangle $ is a $4$-generated symmetric semigroup, not a complete intersection and $N=\sum_{i=1}^4n_i$, then the Frobenius number is $g(S)=\alpha_1n_1+\alpha_{32}n_2+\alpha_4n_4-N$.
\end{cor}

\begin{proof} By Corollary \ref{cor4}, we obtain $a_1+\alpha_1n_1=\alpha_1n_1+\alpha_{32}n_2+\alpha_4n_4$ and thus we have $\beta_3=\beta_{3,\alpha_1n_1+\alpha_{32}n_2+\alpha_4n_4}$ by Proposition \ref{bettisoc}.
 \end{proof}

\begin{ex}
Let  $S= \langle7,9,8,13\rangle$. Then, by Theorem $3$, $S$  is $4$-generated symmetric as 
$$\alpha_{13}=\alpha_{14}=\alpha_{24}=\alpha_{31}=\alpha_{32}=\alpha_{43}=1,\alpha_{21}=\alpha_3=\alpha_4=\alpha_{42}=2,\alpha_1=\alpha_2=3.$$ 
We compute $(a_1,a_2,a_3,a_4,a_5)=(35,29,40,30,34)$ and get
$$g(S)=a_1+\alpha_1n_1-N=\alpha_4n_4+\alpha_{32}n_2+\alpha_1n_1-N=21+9+26-37=19.$$  Indeed, we can determine $S$ completely by appealing to Hilbert series.
\begin{eqnarray*}\mbox{So,}\; \sum_{s\in S}z^s&=&\frac{1-z^{21}-z^{27}-z^{16}-z^{26}-z^{22}+z^{35}+z^{29}+z^{40}+z^{30}+z^{34}-z^{56}}{(1-z^7)(1-z^9)(1-z^8)(1-z^{13})}\\&=&1+z^7+z^8+z^9+z^{13}+z^{14}+z^{15}+z^{16}+z^{17}+z^{18}+z^{20}/(1-z).
\end{eqnarray*}
Therefore, $S=\{0,7,8,9,13,14,15,16,17,18\} \cup \{s \in \mathbb{Z} : s \geq 20\}$.
\end{ex}

\begin{cor}\label{corpseudo} If $S$ is a $4$-generated pseudosymmetric semigroup, we always have
$$b_1=\alpha_1n_1+n_2=\alpha_3n_3+(\alpha_{21}+1)n_1=n_2+n_3+(\alpha_4-1)n_4$$
$$b_3=\alpha_1n_1+(\alpha_3-1)n_3=\alpha_3n_3+(\alpha_4-1)n_4=(\alpha_1-\alpha_{21}-1)n_1+n_2+(\alpha_4-1)n_4$$
$$b_4=\alpha_2n_2+n_1+(\alpha_3-1)n_3=\alpha_4n_4+n_2=(\alpha_{21}+1)n_1+(\alpha_3-1)n_3+n_4$$
$$b_5=\alpha_1n_1+n_4=(\alpha_1-\alpha_{21})n_1+\alpha_2n_2=n_1+(\alpha_2-1)n_2+\alpha_3n_3=n_3+\alpha_4n_4$$
$$b_6=\alpha_2n_2+(\alpha_4-1)n_4=\alpha_{21}n_1+\alpha_4n_4=(\alpha_{21}+1)n_1+(\alpha_2-1)n_2+(\alpha_3-1)n_3$$
and
\begin{eqnarray*}c_1&=&b_1+n_4=b_3+(\alpha_{21}+1)n_1=b_4+n_3=b_5+n_2=b_6+\alpha_3n_3\\
c_2&=&b_1+(\alpha_2-1)n_2+(\alpha_3-1)n_3=\alpha_2n_2+\alpha_3n_3+(\alpha_4-1)n_4=b_3+\alpha_2n_2\\
&=&b_4+(\alpha_1-1)n_1=b_5+\alpha_{21}n_1+(\alpha_3-1)n_3=b_6+\alpha_3n_3.
\end{eqnarray*}
\end{cor}

 \begin{proof} This follows from the different ways to determine the degrees of $H_2({\bf F})$ and $H_3({\bf F})$ in the resolution ${\bf F}$
 in the same way as in the proof of Corollary~\ref{cor4}. So, if we multiply $\phi_1$ with the first column of $\phi_2$ we get $f_1x_2+f_3x_1^{\alpha_{21}+1}-f_5x_3$ whose degree $b_1$ is the inner degree where we have $\beta_{2,b_1}=1$. We did not include $b_2=\deg(f_2)+\deg(f_3)$ as it does not give any new relation. The last two numbers $c_1$ and $c_2$ give the inner degrees of the last free module $A(-c_1) \oplus A(-c_2)$ in the resolution.
  \end{proof}
  
  \begin{cor} If $S=\langle n_1,\dots,n_4\rangle $ is a $4$-generated pseudosymmetric semigroup, then
$PF(S)=\{ \alpha_1 n_1+n_2+n_4-N,\alpha_1n_1+\alpha_2 n_2+(\alpha_3-1)n_3-N\}$, where   $N=\sum_{i=1}^4n_i$.
\end{cor}

\begin{proof} By Proposition \ref{bettisoc}, we know that $PF(S)=\{c_1-N,c_2-N\}$, where the numbers $c_1=\alpha_1n_1+n_2+n_4$ and 
$c_2=\alpha_1n_1+\alpha_2 n_2+(\alpha_3-1)n_3$ by Corollary \ref{corpseudo}.
 \end{proof}

\begin{ex} \label{pseudoex} Let $S=\langle 13,9,11,14\rangle$. Then $\alpha_1=\alpha_2=\alpha_4=3,\alpha_3=2,\alpha_{21}=1$ and $S$ is $4$-generated pseudosymmetric. We get $(b_1,b_2,b_3,b_4,b_5,b_6)=(48,49,50,51,53,55)$,
$PF(S)=\{39+9+14-47,39+11+27-47\}=\{15,30\}$ and

$\sum_{s\in S}z^s=(1-z^{39}-z^{27}-z^{22}-z^{42}-z^{37}+z^{48}+z^{49}+z^{50}+z^{51}+z^{53}+z^{55}-z^{62}-z^{77})/((1-z^{13})(1-z^9)(1-z^{11})(1-z^{14}))=
1+z^9+z^{11}+z^{13}+z^{14}+z^{18}+z^{20}+z^{22}+z^{23}+z^{24}+z^{25}+z^{26}+z^{27}+z^{28}+z^{29}+z^{31}/(1-z)$.
\end{ex}

\section{Strongly indispensable minimal Free Resolutions}
Motivated by the open questions listed at the end of \cite{haraJA}, our main aim here is to classify numerical semigroup rings with small embedding dimensions whose minimal free resolutions are strongly indispensable.  Toric ideals generated minimally by indispensable binomials or equivalently those having a unique minimal generating set are of special importance for some emerging problems arising from Algebraic Statistics, see e.g. \cite{tak}. We recall briefly that indispensable binomials are constant multiples of those binomials that appear in every minimal \textit{binomial} generating set. Strongly indispensable binomials are those whose constant multiples are present in \textit{every} (not necessarily binomial) minimal generating set. Similarly, one can talk about (strong) indispensability of higher syzygies, by requiring that they must be present in every (not necessarily simple) minimal free resolution, see \cite{haraCM,haraJA} for more technicalities. It follows from \cite[Corollary 4.2]{haraCM} that the two concepts coincide for $0$-syzygies (binomial generators), since clearly every strongly indispensable $i$-syzygy is indispensable. Nevertheless, not every indispensable $i$-syzygy is strongly indispensable for $i>0$ by \cite[Example 6.5]{haraJA}.

For a graded minimal free $A$-resolution
$${\bf F}: \ 0\longrightarrow A^{\beta_{k-1}}\stackrel{\phi_{k-1}}{\longrightarrow}A^{\beta_{k-2}}\stackrel{\phi_{k-2}}{\longrightarrow}\cdots\stackrel{\phi_2}{\longrightarrow}A^{\beta_1}\stackrel{\phi_1}{\longrightarrow}A^{\beta_0}{\longrightarrow}K[S]{\longrightarrow}0 $$
of $K[S]$, let $A^{\beta_{i}}$ be generated in degrees $s_{i,j}\in S$, which we call  $i$-Betti degrees, i.e. $\displaystyle A^{\beta_{i}}=\bigoplus_{j=1}^{\beta_i} A[{-s_{i,j}}]$.

The resolution $({\bf F},\phi)$ is strongly indispensable if for any graded minimal resolution $({\bf G},\theta)$, we have an injective complex map
$i\colon({\bf F},\phi)\longrightarrow({\bf G},\theta)$.
The following will be very useful for accomplishing the classification of numerical semigroups of small embedding dimensions whose minimal free resolutions are strongly indispensable.  We consider the partial order on $S$ given by  $s_1\succ_S s_2$ if $s_1-s_2 \in S$.

\begin{lm}\label{indispensable} A minimal graded free resolution of $K[S]$ is strongly indispensable if and only if the differences between the $i$-Betti degrees do not belong to $S$ for all $i=1,\dots,k-1$.
\end{lm}

\begin{proof} This follows from Theorem $4.7$ in \cite{haraCM}, since, for each $i$, the differences of $i$-Betti degrees do not belong to the semigroup $S$ if an only if all  $i$-Betti degrees are different and minimal with respect to $\succ_S$ if and only if the corresponding graded Betti numbers are one and all $i$-Betti degrees are minimal with respect to $\succ_S$.
\end{proof}

\begin{ex} If $S$ is $2$-generated,  then trivially $K[S]$ has a strongly indispensable minimal free resolution, because $\beta_0= \beta_1=1$.
 If $S$ is  $3-$generated and not symmetric, then $K[S]$ has a strongly indispensable minimal free resolution, as stated in \cite{haraCM} preceding to Theorem $4.9$, by \cite[Theorem 4.2]{peeva} and \cite[Theorem 4.7]{haraCM}, since the corresponding ideal $I_S$ is \textbf{generic}. 
 \end{ex}

We now single out a special case in which $K[S]$ is Gorenstein.  In this case it
suffices to check the differences of the first half of the indices:

\begin{lm} If $S$ is a symmetric and $k$-generated semigroup, then a minimal graded free resolution of $K[S]$ is strongly indispensable if and only if the differences between the $i$-Betti degrees do not belong to $S$ for  $1\le i \leq (k-1)/2$.
\end{lm}

\begin{proof} It is known by \cite{be} that the resolution is symmetric, i.e. $\beta_i=\beta_{k-1-i}$ and as pointed out by Stanley in the second proof of Theorem $4.1.$ in \cite{sta}, the generators of $A^{\beta_i}$ can be labeled in such a way that their degrees $s_{i,j}$ satisfy $\displaystyle s_{i,j}+s_{k-1-i,\beta_i-j+1}=s_0$, for all $j=1,\dots,\beta_i$, where $s_0=s_{k-1,1}$. Therefore, the differences satisfy $$s_{k-1-i,\beta_i-j+1}-s_{k-1-i,\beta_i-h+1}=s_{i,h}-s_{i,j},$$
 for $1\le i\le (k-1)/2$, completing the proof.
\end{proof}
For $3$ or $4$-generated symmetric semigroups, this reduces our problem to investigate the differences of the 1-Betti degrees, since in this case $\lfloor(k-1)/2\rfloor=1$.

\begin{cor}\label{3-4} If $S$ is a $3$ or $4$-generated symmetric semigroup, then a minimal graded free resolution of $K[S]$ is strongly indispensable if and only if the toric ideal $I_S$ is generated by indispensable binomials. \hfill$\Box$
\end{cor}

Assume that $S=\langle n_1,n_2,n_3 \rangle$ is symmetric. By \cite{he}, up to a permutation of indices, $n_3 \in \langle \frac{n_1}{\a_3},\frac{n_2}{\a_3}\rangle$, where $\a_3=\gcd(n_1,n_2)$. Letting $m_i=\frac{n_i}{\a_3}$ for $i=1,2$, and $n_3=\a_{31}m_1+\a_{32}m_2$ for some non-negative integers $\a_{31}$ and $\a_{32}$, we obtain that $S=\langle \a_3m_1,\a_3m_2,\a_{31}m_1+\a_{32}m_2 \rangle$. In this case, $I_S=(F_1,F_2)$ with $F_1=x_1^{m_2}-x_2^{m_1}$ and $F_2=x_3^{\a_3}-x_1^{\a_{31}}x_2^{\a_{32}}$.

 It is known that complete intersection semigroup rings have minimal free resolutions that are indispensable if and only if differences of the first Betti degrees do not belong to the semigroup, see \cite[Theorem 4.4]{haraCM}. The following extends this to \textit{strongly} indispensable minimal free resolutions.

\begin{prop} \label{prop3ci}Let $S=\langle \a_3m_1,\a_3m_2,\a_{31}m_1+\a_{32}m_2 \rangle$. $K[S]$ has a strongly indispensable minimal free resolution if and only if $(\a_{31},\a_{32})$ is unique with $\dis \a_{31}\a_{32} \neq 0$.
\end{prop}

\begin{proof} By virtue of Corollary \ref{3-4}, it is sufficient to check that the ideal $I_S$ has a unique minimal generating set which is proved in \cite[Theorem 17]{garcia-ojeda}. For the convenience of the reader we give an  alternative proof not using the theory they developed there.

Assume that $\a_{31}\a_{32} = 0$ or $(\a_{31},\a_{32})$ is not unique. We prove then that $|\deg(F_1)-\deg(F_2)| \in S$ so that the minimal free resolution is not (strongly) indispensable by Lemma \ref{indispensable}. Without loss of generality suppose first that $\a_{32}=0$. Then, $F_2=x_3^{\a_3}-x_1^{\a_{31}}$ and thus $|\deg(F_2)-\deg(F_1)|=|\a_{31}-m_2|n_1 \in S$. 

Suppose now that $n_3=\a_{31}m_1+\a_{32}m_2=\a'_{31}m_1+\a'_{32}m_2$, for different tuples $(\a_{31},\a_{32})$ and $(\a'_{31},\a'_{32})$. Then, $F_2$ is not an indispensable binomial, i.e. $\{F_1,F_2\}$ and $\{F_1,x_3^{\a_3}-x_1^{\a'_{31}}x_2^{\a'_{32}}\}$ are two different minimal generating sets for $I_S$. Thus, no minimal free resolution can be indispensable.

For the converse, assume that $(\a_{31},\a_{32})$ is unique with $\dis \a_{31}\a_{32} \neq 0$. In this case, $\a_{31}<m_2$ and $\a_{32}<m_1$, since otherwise $\a_{31}m_1+\a_{32}m_2$ would be $$(\a_{31}-m_2)m_1+(\a_{32}+m_1)m_2\; \mbox{or} \;(\a_{31}+m_2)m_1+(\a_{32}-m_1)m_2.$$ Now $\deg(F_2)-\deg(F_1)=\a_{32}n_2-(m_2-\a_{31})n_1 \in S$ contradicts the fact that $\alpha_2=m_1$ is the least integer with $\alpha_2n_2\in \langle n_1,n_3\rangle$. Similarly, $\deg(F_1)-\deg(F_2)=(m_2-\a_{31})n_1-\a_{32}n_2 \notin S$. Since there is only one $2-$Betti degree, $K[S]$ has a strongly indispensable minimal free resolution by Lemma \ref{indispensable}.
\end{proof}

Since a $3$-generated symmetric semigroup is always complete intersection, using Proposition \ref{prop3ci}, one can easily check when the minimal free resolution is strongly indispensable  as the following illustrates.

\begin{ex} \label{23} Let $\a_3 \geq 2$. Then $S=\langle \a_3\cdot 2,\a_3\cdot3,n_3 \rangle$ has strongly indispensable minimal free resolution if and only if $n_3\in\{5,7\}$.
\end{ex}

Let us now look at symmetric $S=\langle n_1,n_2,n_3,n_4 \rangle$ which is a complete intersection. By \cite{br} or \cite{delorme,ros}, up to a permutation of indices, we have two cases:

Case I: $S'=\langle n_1/\ell,n_2/\ell,n_3/\ell \rangle$ is symmetric as in Proposition \ref{prop3ci}, where the positive integer $\ell={\rm gcd}(n_1,n_2,n_3)$ is the smallest such that $\ell n_4 \in \langle n_1,n_2,n_3\rangle$. If we write $\ell n_4 =\a_{41}n_1+\a_{42}n_2+\a_{43}n_3$ for some non-negative $\a_{41},\a_{42}$ and $\a_{43}$, then $I_S=\langle F_1,F_2,F_3\rangle$, where $F_1$ and $F_2$ are as in Proposition \ref{prop3ci}, and $F_3=x_4^{\ell}-x_1^{\a_{41}}x_2^{\a_{42}}x_3^{\a_{43}}$. Note that $\deg_S(F_i)=\ell \cdot \deg_{S'}(F_i)$, for $i=1,2$. 

Case II: $pp' \in \langle n_1,n_2\rangle \cap \langle n_3,n_4\rangle$ with $p={\rm gcd}(n_1,n_2)$ and $p'={\rm gcd}(n_3,n_4)$. This is equivalent to $p \in \langle \frac{n_3}{p'},\frac{n_4}{p'}\rangle$ and $p' \in \langle \frac{n_1}{p},\frac{n_2}{p}\rangle$. Letting $p=p_3\frac{n_3}{p'}+p_4\frac{n_4}{p'}$ and $p'=p_1\frac{n_1}{p}+p_2\frac{n_2}{p}$ we have $I_S=\langle F_1,F_2,F_3\rangle$, where $F_1=x_1^{\a_{1}}-x_2^{\a_2}$ and $F_2=x_3^{\a_3}-x_4^{\a_4}$ and $F_3=x_1^{p_{1}}x_2^{p_{2}}-x_3^{p_{3}}x_4^{p_4}$, where $\a_i$ is the smallest positive integer such that $\a_in_i \in \langle \{n_1,\dots,n_4\}\setminus \{n_i\}\rangle$, for all $i=1,\dots,4$. 

 The following also extends \cite[Theorem 4.4]{haraCM} to \textit{strongly} indispensable minimal free resolutions.

\begin{tm}\label{thm4ci} Let $S$ be a $4-$generated complete intersection semigroup. Then $K[S]$ has a minimal free resolution that is strongly indispensable if and only if \\
{\rm (Case I)} $(\a_{31},\a_{32})$ is unique with $\dis \a_{31}\a_{32} \neq 0$ and\\
$(\a_{41},\a_{42},\a_{43})$ is unique with at most one of $\a_{41},\a_{42}$ and $\a_{43}$ being zero\\
{\rm (Case II)}  $\dis (p_1,p_2,p_3,p_4)$ is unique with $\dis p_1p_2p_3p_4 \neq 0$.
\end{tm}

\begin{proof} By the virtue of Corollary \ref{3-4} it is enough to prove that the toric ideal $I_S$ is generated by indispensable binomials which we address below.

Case I: We know that $|\deg_{S'}(F_1)-\deg_{S'}(F_2)| \notin S'$ if and only if $(\a_{31},\a_{32})$ is unique with $\dis \a_{31}\a_{32} \neq 0$, by Proposition \ref{prop3ci}. Even though $\deg_S(F_i)=\ell \cdot \deg_{S'}(F_i)$, for $i=1,2$, we still have that $|\deg_{S}(F_1)-\deg_{S}(F_2)| \notin S$ if and only if $(\a_{31},\a_{32})$ is unique with $\dis \a_{31}\a_{32} \neq 0$, since $s' \in S'$ if and only if $\ell \cdot s' \in S$. 

Necessity of uniqueness of $(\a_{41},\a_{42},\a_{43})$ is obvious as in the proof of Proposition \ref{prop3ci}. To prove necessity of the second part, assume that $\a_{41}=\a_{43}=0$. Then, we get $|\deg(F_1)-\deg(F_3)|=| (m_1-\a_{42})n_2|\in S$. The other two cases can be done similarly. 

Now, we prove sufficiency. Set $\a_1=m_2, \a_2=m_1$ and $\a_4=\ell$. By uniqueness, $\a_i>\a_{4i}$, for $i=1,2,3$. This can be seen from the following equations:
\begin{eqnarray*} \ell n_4&=&\a_{41}n_1+\a_{42}n_2+\a_{43}n_3\\
&=&(\a_{41}-\a_1)n_1+(\a_{42}+\a_2)n_2+\a_{43}n_3\\
&=&(\a_{41}+\a_1)n_1+(\a_{42}-\a_2)n_2+\a_{43}n_3\\
&=&(\a_{41}+\a_{31})n_1+(\a_{42}+\a_{32})n_2+(\a_{43}-\a_3)n_3.
\end{eqnarray*}

Now, $\deg(F_3)-\deg(F_1)=\a_4n_4-\a_1n_1 \in S$ implies $\a_4n_4=a_1n_1+a_2n_2+a_3n_3$ with $a_1 \geq \a_1$ which contradicts to $\a_1>\a_{41}$ as by uniqueness $a_1=\a_{41}$.

Permuting $1$ and $3$ in the indices of $\a_i$, $n_i$ and $a_i$ above gives the identical proof for $\deg(F_3)-\deg(F_2)\notin S$.

Since at least two of the numbers $\a_{41}, \a_{42}$ and $\a_{43}$ are non-zero, either $\a_{41}\neq 0$ or $\a_{42}\neq 0$. As $\deg(F_1)-\deg(F_3)=\a_1n_1-\a_{41}n_1-\a_{42}n_2-\a_{43}n_3 \in S$ implies $(\a_1-\a_{41})n_1\in \langle n_2,n_3,n_4\rangle$, when $\a_{41}\neq 0$, and yields $(\a_2-\a_{42})n_2\in \langle n_1,n_3,n_4\rangle$, when $\a_{42}\neq 0$, both cases contradict to $\a_i$ being the smallest, for $i=1,2$. 

Let $\deg(F_2)-\deg(F_3)=\a_3n_3-\a_{41}n_1-\a_{42}n_2-\a_{43}n_3=\sum_{i=1}^{4}u_in_i \in S$. Since $\a_3$ is the smallest positive integer with $\a_3n_3\in \langle n_1,n_2,n_4\rangle$, it follows that $\a_{43}=u_3=0$. Setting $\gamma _1=\a_{31}-\a_{41}-u_1$ and $\gamma _2=\a_{32}-\a_{42}-u_2$, and by using $\a_3n_3=\a_{31}n_1+\a_{32}n_2$, we get the equality 
\begin{equation} \label{gamma} \gamma _1n_1+\gamma _2n_2=u_4n_4.
\end{equation} Since $u_4n_4\geq 0$, both $\gamma_1$ and $\gamma_2$ can not be negative simultaneously. So, if $\gamma _1<0$ and $\gamma _2 \geq 0$, then equation \ref{gamma} yields $$\gamma_2 n_2=-\gamma_1n_1+u_4n_4 \in \langle n_1,n_3, n_4\rangle,$$ which contradicts to $\a_2$ being the smallest coefficients of $n_2$ with this property, since by the proof of Proposition \ref{prop3ci}, we have $\a_2=m_1>\a_{32}>\gamma_2$. The case $\gamma _2<0$ and $\gamma _1 \geq 0$ is taken care of the same way. Let us analyze   
the case where both $\gamma_1$ and $\gamma_2$ are non-negative, in which case $u_4n_4 \in \langle n_1,n_2,n_3 \rangle$. Since $\a_4$ is the smallest, we must have $u_4 \geq \a_4$. Suppose $u_4=q\a_4+r$ with $0<q$ and $0\leq r <\a_4$. Then, equation \ref{gamma} becomes $\gamma'_1n_1+\gamma'_2n_2=r n_4$, where $\gamma'_1=\a_{31}-(q+1)\a_{41}-u_1$ and $\gamma'_2=\a_{32}-(q+1)\a_{42}-u_2$. Similar arguments as above leads to the case where both $\gamma'_1$ and $\gamma'_2$ are non-negative, which is a contradiction as the coefficient of $n_4$, i.e. $r$, is smaller than $\a_4$ now.

Case II: Necessity of the condition that $\dis (p_1,p_2,p_3,p_4)$ is unique with $\dis p_1p_2p_3p_4 \neq 0$ is easy to check. If $\dis (p'_1,p'_2,p'_3,p'_4)$ is another quadruple, then replacing $F_3$ by $x_1^{p'_{1}}x_2^{p'_{2}}-x_3^{p'_{3}}x_4^{p'_4}$, we obtain two different minimal generating set for the ideal $I_S$. On the other hand, if $p_1=0$, then $\deg(F_1)-\deg(F_3)=(\a_2-p_2)n_2 \in S$, since $\a_i>p_i$ by \cite[Theorem $4$]{br}, for all $i=1,\dots,4$. The other three cases are similar.

Now, we prove sufficiency. If $\deg(F_3)-\deg(F_1) \in S$, then $p_1n_1+p_2n_2-\a_{1}n_1=\sum_{i=1}^{4}u_in_i $, for some non-negative integers $u_i$, which implies that
$$(p_2-u_2)n_2=(\a_{1}-p_1+u_1)n_1+u_3n_3+u_4n_4.$$ As $\a_{1}-p_1>0$, the right hand side is positive and so is $p_2-u_2$. But, this  contradicts to $\a_2$ being the smallest, since $0<p_2-u_2\leq p_2 <\a_2$. One can easily prove that $\deg(F_3)-\deg(F_2) \notin S$ in a similar fashion.

Assume now that $\deg(F_2)-\deg(F_3)\in S$ and so $\a_{3}n_3-p_3n_3-p_4n_4=\sum_{i=1}^{4}u_in_i $, for some non-negative integers $u_i$. Then, we get that $$(\a_{3}-p_3-u_3)n_3=u_1n_1+u_2n_2+(p_4+u_4)n_4>0$$ and that $\a_{3}-p_3-u_3>0$ contradicting to $\a_3$ being the smallest. $\deg(F_1)-\deg(F_3)\notin S$ can be shown similarly.

Suppose finally that $\deg(F_2)-\deg(F_1)\in S$, i.e. $\a_{3}n_3-\a_1n_1=\sum_{i=1}^{4}u_in_i $, for some non-negative integers $u_i$. Then, we obtain $$\a_{3}n_3-u_3n_3-u_4n_4=\a_1n_1+u_1n_1+u_2n_2 \in \bZ\{n_1,n_2\} \cap \bZ\{n_3,n_4\}=\bZ\{pp'\}.$$ Let $\delta pp'=\a_1n_1+u_1n_1+u_2n_2$. Since $\a_1>0$, we have clearly $\delta >0$. Then, the other equality $\a_{3}n_3-u_3n_3-u_4n_4=\delta pp'$ yields the following
$$\deg(F_2)-\deg(F_3)=\a_{3}n_3-pp'=u_3n_3+u_4n_4+(\delta-1) pp' \in S,$$
which we prove absurd in the previous step.
\end{proof}

\begin{ex}  Let $\ell \ge2$. It is easy to see, using Example \ref{23} and Theorem \ref{thm4ci}, that a complete intersection semigroup of the type $S=\langle \ell \cdot 4,\ell \cdot 6,\ell \cdot 5,n_4 \rangle$ has  strongly indispensable minimal free resolution if and only if $n_4\in \{9,11,13\}$. An example for the second case of the previous theorem is produced by the semigroup $$S=\langle 15\cdot 5, 15\cdot 12, 17\cdot 7,17\cdot 8 \rangle.$$
Since $F_3=x_1x_2-x_3x_4$ corresponds to the unique quadruple $(p_1,p_2,p_3,p_4)=(1,1,1,1)$, the minimal free resolution is strongly indispensable.
\end{ex}

\begin{tm}\label{thm4nci} Let $S$ be a symmetric, $4-$generated but not a complete intersection semigroup. Then, the minimal free resolution of $K[S]$ is strongly indispensable.
\end{tm}

\begin{proof} As before Corollary \ref{3-4} reduces the problem to prove that the toric ideal $I_S$ is generated by indispensable binomials which is addressed in \cite[Corollary 3.13]{koj}. As an alternative, we give a direct proof here. 

Assume that $\deg (f_3)-\deg (f_{1}) \in S$, where $f_i$ are as in Theorem \ref{4generated}. Then, by using B in Corollary \ref{cor4} also, we obtain the following $$\a_3n_3-\a_1n_1=\a_{43}n_3-\a_{14}n_4=u_1n_1+u_2n_2+u_3n_3+u_4n_4,$$ for some non-negative integers $u_i$. This means that $$(\a_{43}-u_3)n_3=u_1n_1+u_2n_2+(\a_{14}+u_4)n_4$$ which contradicts to $\a_3$ being the smallest positive integer $u$ satisfying $un_3\in \langle n_1,n_2,n_4 \rangle$, since $\a_3>\a_{43}$ by Theorem \ref{4generated}. One can similarly check the others.
\end{proof}

For numerical semigroups of embedding dimension less than $5$, we have seen that
(strong) indispensability of a minimal free resolution of a Gorenstein $K[S]$ depends only on the first Betti degrees, that is differences of all Betti degrees depend only on the differences of the first Betti degrees. This is no longer true in embedding dimension $5$ as the following illustrates.

\begin{ex} Take the symmetric semigroup $S=\langle 19,27,28,31,32\rangle$. One can see that its $i$-Betti degrees are as follows:\\
\noindent$i=1: 59, 81, 82, 83, 84, 85, 88, 89, 92, 93, 94, 95, 96\\$
$i=2: 109, 110, 111, 112, 113, 113, 114, 115, 115, 116, 116, 117, 119,\\ 120, 120, 121, 121, 122, 123, 123, 124, 125, 126, 127$\\
$i=3: 140, 141, 142, 143, 144, 147, 148, 151, 152, 153, 154, 155, 177$\\
$i=4: 236$

Since the differences of the first Betti degrees do not belong to $S$, we see that $I_S$ is generated by indispensable binomials but it can not have a strongly indispensable minimal free resolution as $\beta_{2,113}=2$  and thus $113-113=0 \in S$.
\end{ex}

\begin{prop} Let $S=\langle n_1,n_2,n_3,n_4 \rangle$ be a $4$-generated pseudosymmetric semigroup. Then $K[S]$ has a strongly indispensable resolution if and only if  the
differences between the 1-Betti degrees and  between the 2-Betti degrees do not belong to $S$.
\end{prop}
\begin{proof} Let $N= \sum_{i=1}^4 n_i$. By Proposition \ref{bettisoc} the 3-Betti degrees are $g(S)/2+N$ and $g(S)+N$. Thus in this case the difference between the 3-Betti degrees, $g(S)/2$, is never in $S$.
\end{proof}
 The two conditions in the proposition above are necessary as the following examples reveal.
\begin{ex}
The semigroup $S=\langle 13,9,11,14\rangle$ of example~\ref{pseudoex} does not have a strongly indispensable resolution, since non-negative differences between the $1$-Betti degrees are $2,3,5,10,12,15,17,20$ but $20\in S$, even though non-negative differences between the $2$-Betti degrees, i.e. $1,2,3,4,5,6,7$, do not belong to $S$.

 For $(\a_1,\a_2,\a_3,\a_4,\a_{21})=(5,2,2,2,2)$, we get by Theorem \ref{kom} that $S=\langle5,12,11,14\rangle$ is pseudosymmetric. Non-negative differences of the first Betti degrees are $1,2,3,4,6$ which do not belong to $S$ and thus $I_S$ is generated minimally by indispensable binomials. But non-negative differences of the second Betti degrees are $1,2,3,4,6,7,8,9,10$ and since $10 \in S$ it follows that the minimal free resolution of $K[S]$ is not strongly indispensable.

 Finally, $S=\langle5,11,8,12\rangle$ is pseudosymmetric again by Theorem \ref{kom}, where the numbers $n_i$ are determined by $(\a_1,\a_2,\a_3,\a_4,\a_{21})=(4,2,2,2,2)$. It can be seen that non-negative differences of the first Betti degrees are $1,2,3,4,6,7,8$ with $8\in S$, and non-negative differences of the second Betti degrees are $1,3,4,6,7,10$ with $10\in S$. Therefore, both conditions do not hold and $I_S$ does not have a strongly indispensable minimal free resolution.
\end{ex}

\section*{Acknowledgements} We would like to thank the referee for very helpful
suggestions. Most of the examples are computed by using the computer algebra system Macaulay $2$, see \cite{Mac2}.

\bigskip

\noindent Valentina Barucci, University of Rome 1, barucci@mat.uniroma1.it

\medskip

\noindent Ralf Fr\"oberg, University of Stockholm, ralff@math.su.se

\medskip

\noindent Mesut \c{S}ahin, \c{C}ank{\i}r{\i} Karatekin University, mesutsahin@karatekin.edu.tr

\begin{thebibliography}{Dillo83}

\bibitem{br} H.~Bresinsky, {\it Symmetric semigroups of integers generated by 4 elements}, Manuscripta Math. {\bf 17} (1975), 205--219.

\bibitem{bu-ei} D.~Buchsbaum, D.~Eisenbud, {\it What makes a complex exact?}, J. Algebra {\bf 25} (1973), 259--268.

\bibitem{be} D. A. Buchsbaum and D. Eisenbud, \textit{Algebra structures for finite free resolutions,
and some structure theorems for ideals of codimension $3$}, Amer. J. Math. 99 (1977),
447--485.

\bibitem{haraCM} H.~Charalambous, A.~Thoma, {\it On
simple A-multigraded minimal resolutions},
Contemporary Mathematics {\bf 502} (2009), 33--44.

\bibitem{haraJA} H.~Charalambous, A.~Thoma, {\it On the generalized Scarf complex of lattice ideals},
J. Algebra {\bf 323} (2010), 1197--1211.

\bibitem{delorme} C.~Delorme, {\it Sous-mono\"{\i}des d'intersection compl\`{e}te de $N$},
Ann. Sci. \'{E}cole Norm. (4) 9 No.1 (1976), 145-154.

\bibitem{den} G.~Denham, {\it Short generating functions for some semigroup algebras}, Electr. J. Combinatorics {\bf 10}  (2003), R36.

\bibitem{ei} D.~Eisenbud, {\it Commutative algebra. With a view toward algebraic geometry}, Graduate Texts in Mathematics, {\bf 150},
Springer Verlag, New York 1995.

\bibitem{garcia-ojeda} P.~A.~Garcia-Sanchez, I. Ojeda, {\it Uniquely presented finitely generated commutative monoids},
Pac. J. Math. {\bf 248}, No. 1 (2010), 91--105.

\bibitem{Mac2} D. Grayson, M. Stillman, Macaulay 2--A System for Computation in Algebraic Geometry and Commutative Algebra. http://www.math.uiuc.edu/Macaulay2.

\bibitem{he} J.~Herzog, {\it Generators and relations of abelian semigroups and semigroup rings}, Manuscripta Math. {\bf 3} (1970), 175--193.

\bibitem{koj} A. Katsabekis, I. Ojeda, {\it An indispensable classification of monomial curves in
$\mathbb{A}^4 $}, http://arxiv.org/pdf/1103.4702.pdf

\bibitem{ko} J.~Komeda, {\it On the existence of Weierstrass points with a certain semigroup generated by 4 elements}, Tsukuba J. Math. {\bf 6} (1982), 237--279.

\bibitem{ku} E.~Kunz, {\it The value-semigroup of a one-dimensional Gorenstein ring}, Proc. Amer. Math. Soc. {\bf 25} (1970), 748--751.

\bibitem{peeva} I. Peeva, B. Sturmfels, {\it Generic lattice ideals}, J. Amer. Math. Soc. {\bf 11} (1998), 363--373.

\bibitem{ros} J. C. Rosales, {\it On presentations of subsemigroups of $\mathbb{N}^n$},
Semigroup Forum {\bf 55} (1997), 152-159.

\bibitem{ro-ga}  J.~C.~Rosales, P.~A~Garcia-Sanchez, {\it Numerical semigroups with embedding dimension three}, Arch. Math. {\bf 83} (2004), 488--496.

\bibitem{sta} R. P. Stanley, \textit{Hilbert functions of graded algebras}, Adv. Math. {\bf 28}
(1978), 57--83.

\bibitem{tak} A. Takemura and S. Aoki, \textit{Some characterizations of minimal Markov basis for sampling from discrete conditional distributions}, Ann. Inst. Statist. Math. 56:1 (2004), 1?17.

\end{thebibliography}
\end{document}